\documentclass[submission]{dmtcs}
\usepackage[latin1]{inputenc}
\usepackage{amsmath,amssymb,mathrsfs}

\newtheorem{prop}{Proposition}[section]
\newtheorem{coro}[prop]{Corollary}

\newtheorem{lem}[prop]{Lemma}
\newtheorem{rem}[prop]{Remark}
\newtheorem{defn}[prop]{Definition}

\makeatletter
\@addtoreset{equation}{section}
\makeatother

\newcommand\C{\mathbb{C}}

\newcommand\Sc{\mathcal{S}}

\renewcommand{\geq}{\geqslant}
\renewcommand{\leq}{\leqslant}
\def\D{\mathcal{D}}

\usepackage{color}
\definecolor{darkgreen}{rgb}{0,0.4,0}
\definecolor{MyDarkBlue}{rgb}{0,0.08,0.50}
\definecolor{BrickRed}{rgb}{0.65,0.08,0}

\hypersetup{
colorlinks=true,       
    linkcolor=blue,          
    citecolor=red,        
    filecolor=BrickRed,      
    urlcolor=darkgreen        
}

\begin{document}

\author{Guy Fayolle\addressmark{1}\thanks{Email: \email{Guy.Fayolle@inria.fr}}\and%
        Kilian Raschel\addressmark{2}\thanks{Email: \email{Kilian.Raschel@lmpt.univ-tours.fr}}}
\title[Asymptotics in the counting of walks in the quarter plane]{Some exact asymptotics in the counting of walks in the quarter plane}
\address{\addressmark{1} INRIA Paris-Rocquencourt (France) \\
\addressmark{2}CNRS \& LMPT, Universit\'e Fran\c cois Rabelais  (France)}

\received{30 January 2012}
\revised{\today}
\accepted{April 2012}
\maketitle
\begin{abstract}Enumeration of planar lattice walks is a classical topic in combinatorics, at the cross-roads of several domains (e.g., probability, statistical physics, computer science). The aim of this paper is to propose a new approach to obtain some exact asymptotics for walks confined to the quarter plane.
\end{abstract}
\emph{AMS $2000$ Subject Classification: primary 60G50; secondary 30F10, 30D05}

\clearpage
\section{Introduction}
Enumeration of planar lattice walks is a most classical topic in combinatorics. For a given set 
$\Sc$ of admissible steps (or jumps), it is a matter of counting the number of paths of a certain length, which start and end at some arbitrary points, and might even be restricted to some region of the plane.  Then three natural important questions arise.
\begin{enumerate}
     \item \label{question-howmany} How many such paths do exist?
     \item \label{question-asymptotic} What is the asymptotic behavior, as their length goes to infinity, of the number of walks ending at some given point or domain (for instance one axis)? 
     \item \label{question-nature} What is the nature of the generating function of the numbers of walks? Is it \emph{holonomic}\footnote{A function of several complex variables is said to be holonomic if the vector space over the field of rational functions spanned by the set of all derivatives is finite dimensional. In the case of one variable, this is tantamount to saying that the function is solution of a linear differential equation where the coefficients are rational functions (see \cite{FLAJ}).}, and, in that case, \emph{algebraic} or even 
     \emph{rational}?  
\end{enumerate}

If the paths are not restricted to a region, or if they are constrained to remain in a half-plane, it turns out \cite{MBM2} that the generating function has an explicit form (question \ref{question-howmany}), and is, respectively, rational or algebraic (question~\ref{question-nature}). Question \ref{question-asymptotic} can then be solved from the answer to \ref{question-howmany}.

The situation happens to be much richer if the walks are confined to the quarter plane $\mathbb{Z}_+^2$. As an illustration, let us recall that some walks admit an algebraic generating function, see \cite{Gessel}  for Kreweras' walk (see Figure \ref{ExExEx}), while others admit a generating function which is not even holonomic, see \cite{MBM2} for the so-called \emph{knight walk}. 

In the sequel, we focus on  walks confined to $\mathbb{Z}_+^2$, starting at the origin and having \emph{small steps}. This means exactly that  the set $\mathcal S$ of admissible steps is included in the  set of the eight nearest neighbors, i.e., $\mathcal S \subset \{-1, 0, 1\}^2 \setminus\{(0, 0)\}$. By using  an extended Kronecker's delta, we shall write
\begin{equation}
\label{def_deltaij}
     \delta_{i,j}=\left\{\begin{array}{ccc}
     1 &\text{if} & (i,j)\in\mathcal S,\\
     0 &\text{if} & (i,j)\notin\mathcal S.
     \end{array}\right.
\end{equation}
On the boundary of the quarter plane, allowed jumps are the natural ones: steps that would take the walk out of $\mathbb{Z}_+^2$ are obviously discarded, see examples on Figure \ref{ExExEx}. 

\begin{figure}[t]
\begin{center}
\begin{picture}(415.00,68.00)
\includegraphics{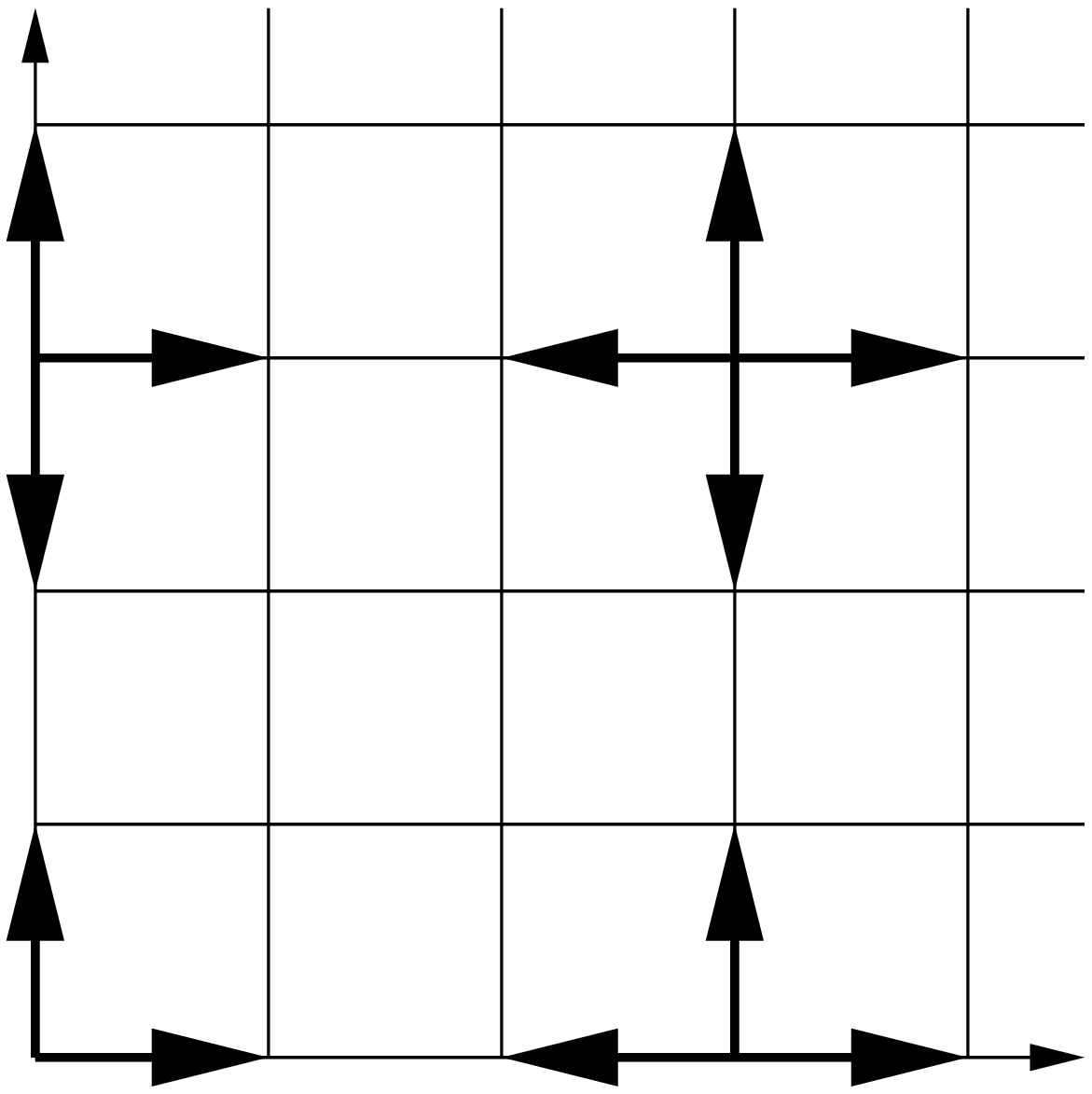}
\hspace{38mm}
\includegraphics{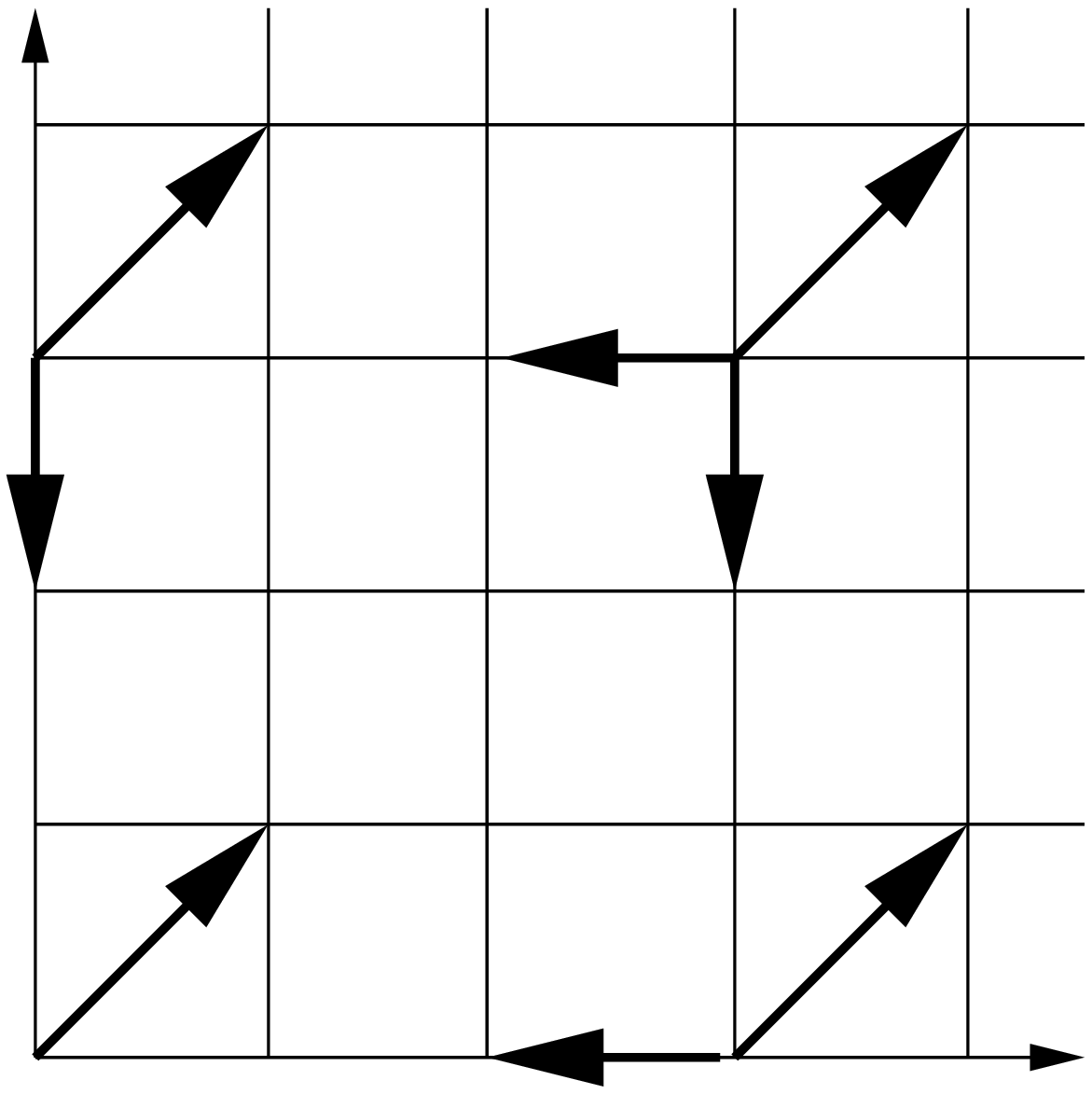}
\hspace{38mm}
\includegraphics{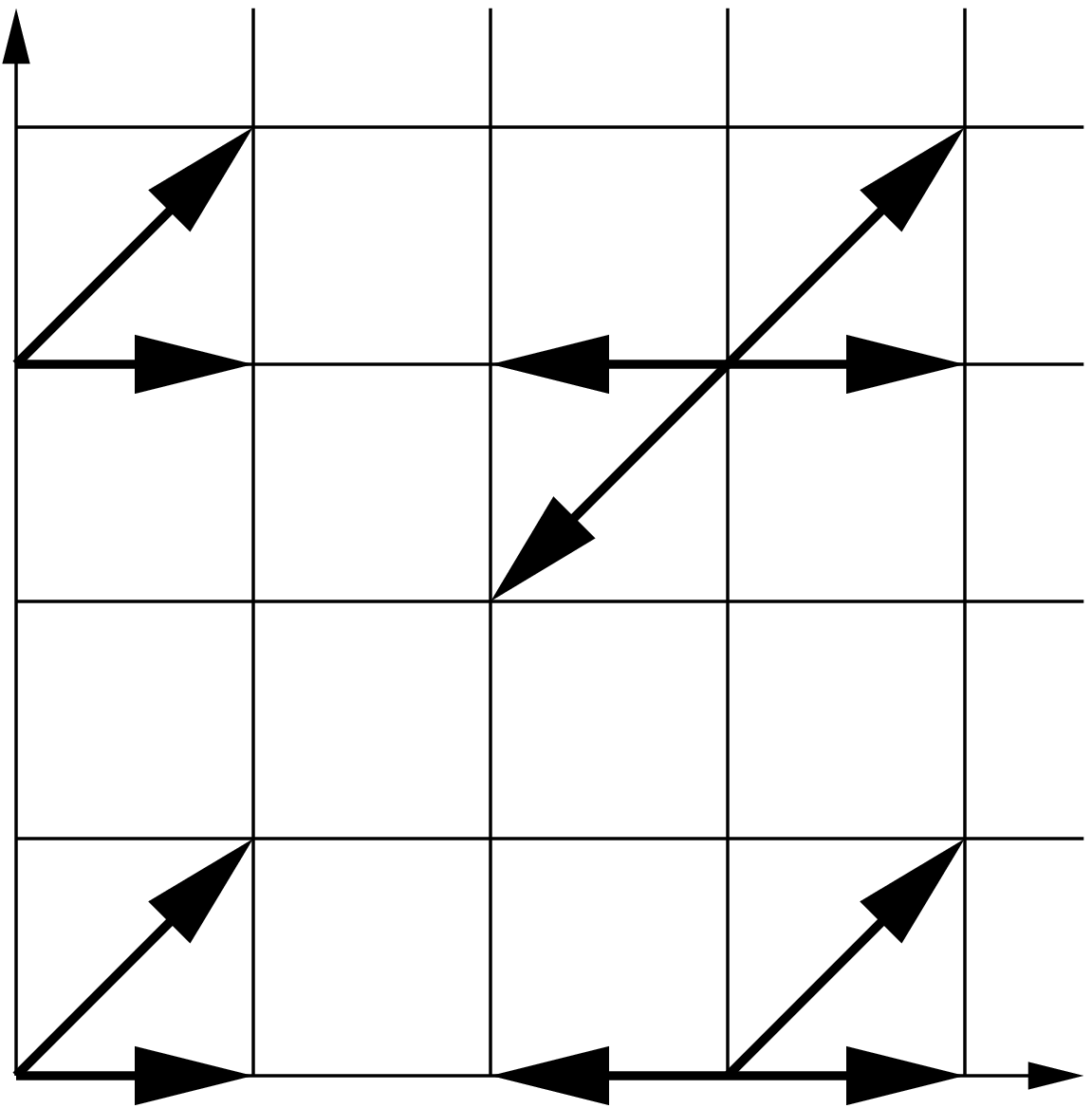}
\hspace{38mm}
\includegraphics{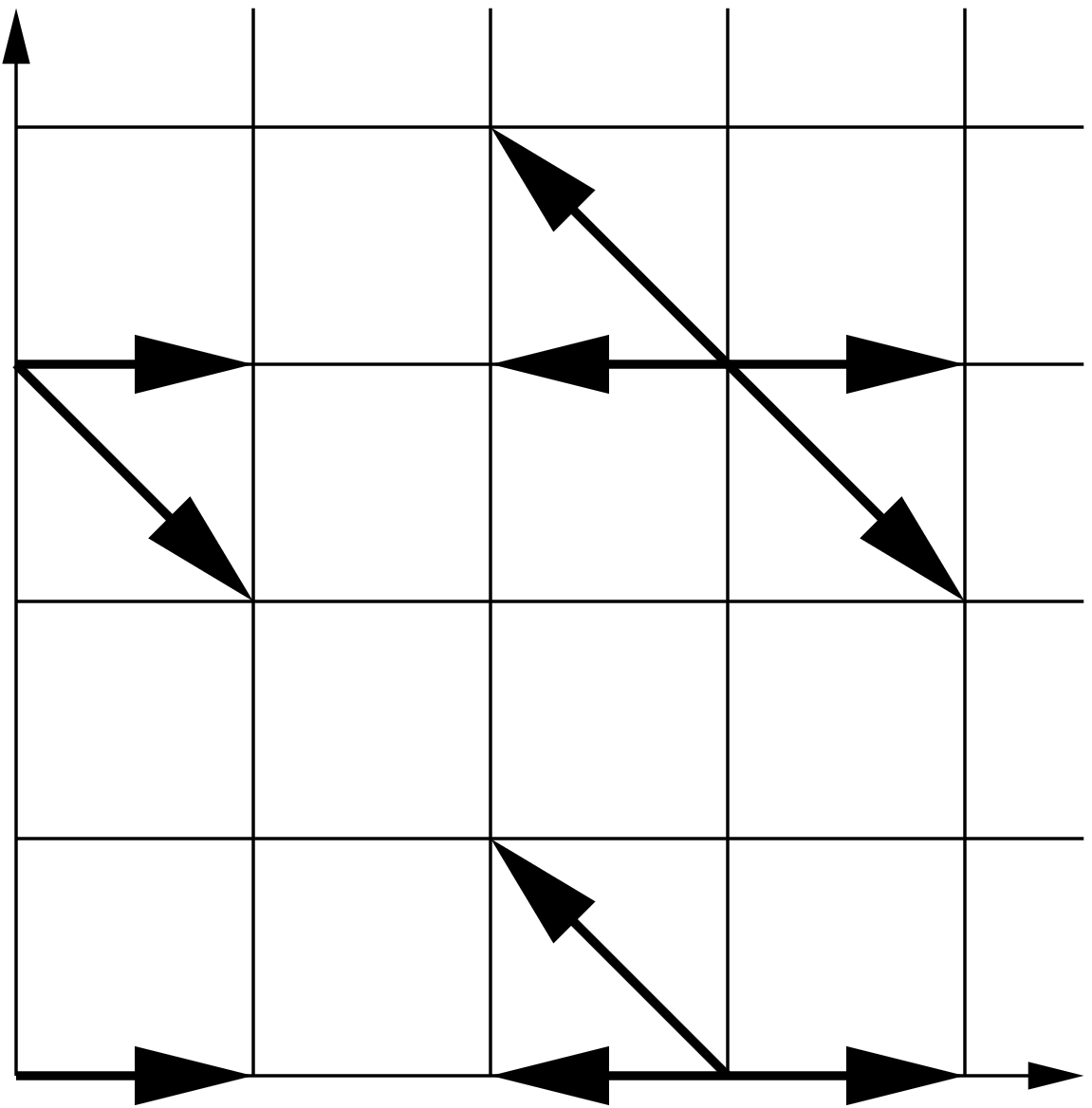}
\end{picture}
\end{center}
\caption{Four famous examples, known as the simple, Kreweras', Gessel's and Gouyou-Beauchamps' walks, respectively}
\label{ExExEx}
\end{figure}

There are $2^8$ such models. Starting from the existence of simple geometrical symmetries, Bousquet-M\'elou and Mishna \cite{BMM} have shown that there are in fact $79$ types of essentially distinct walks---we will often refer to these $79$ walks tabulated in \cite{BMM}. For each of these $79$ models,  $q(i,j,n)$ will denote \emph{the number of paths confined to $\mathbb{Z}_{+}^{2}$, starting at $(0,0)$ and ending at $(i,j)$ after $n$ steps}. The associated generating function will be written
\begin{equation} 
\label{eq:def_Q}
     Q(x,y,z)=\sum_{i,j,n\geq 0}q(i,j,n)x^{i}y^{j}z^{n}.
\end{equation}

Answers to questions \ref{question-howmany} and \ref{question-nature} have been recently determined for all $79$ models of walks. To present these results, we need to define a certain group, introduced in a probabilistic context in \cite{MAL}, and called the {\em group of the walk}. This group $W=\langle \Psi,\Phi\rangle$ of birational transformations in $\mathbb C^2$ leaves invariant the function $\sum_{-1\leq i,j\leq 1}\delta_{i,j}x^{i}y^{j}$, and has the following two generators:
     \begin{equation}
     \label{def_generators_group}
          \Psi(x,y)=\left(x,\frac{\sum_{-1\leq i\leq 1}\delta_{i,-1}x^{i}}
          {\sum_{-1\leq i\leq 1}\delta_{i,+1}x^{i}}\frac{1}{y}\right),
          \qquad \Phi(x,y)=\left(\frac{\sum_{-1\leq j\leq 1}\delta_{-1,j}y^{j}}
          {\sum_{-1\leq j\leq 1}\delta_{+1,j}y^{j}}\frac{1}{x},y\right).
     \end{equation}
Clearly $\Psi\circ\Psi=\Phi\circ\Phi=\text{id}$, and $W$, when it is finite, is a dihedral group, whose order is even and at least four. The order of $W$ is calculated in \cite{BMM} for each of the $79$ cases: $23$ walks admit a finite group (then of order $4,6$ or $8$), and the $56$ others have an infinite group.

The expression of the generating function $Q(x,y,z)$---or equivalently the answer to \ref{question-howmany}---for the $23$ walks with a finite group has been determined in \cite{BK0,BMM}, mainly by algebraic manipulations, starting from the so-called \emph{kernel} method in the two-dimensional case (see, e.g., \cite{FIM} and references therein).

As for the $56$ models with an infinite group,  the function $Q(x, y, z)$ related to the $5$ 
\emph{singular} walks (i.e., walks having no jumps to the West, South-West and South) was found in \cite{MM} by using elementary manipulations on \eqref{counting_func_eq}. For the remaining  $51$ non-singular walks, $Q(x, y, z)$ was found in \cite{Ra} by the method, initiated in \cite{FI,FIM}, of reduction to boundary value problems (BVP). 

Concerning \ref{question-nature}, the answer was obtained in \cite{BK0,BMM,FR1} for the $23$ walks with a finite group: the function $Q(x,y,z)$ is always holonomic, and even algebraic when 
$\sum_{-1\leq i,j\leq1}i j\delta_{i,j}$ is positive.  For the $56$ models with an infinite group, it was also proved that they all admit a non-holonomic generating function: see \cite{MM} for the $5$ singular walks, and \cite{KuRa2} for the non-singular ones.

Let us now focus on question \ref{question-asymptotic}, which actually is the main topic of this paper. 
A priori, it has a link with question \ref{question-howmany} (indeed, being provided with an expression for the generating function, one should hopefully be able to deduce asymptotics of its coefficients), and with \ref{question-nature} as well (since holonomy is strongly related to singularities, and hence to the asymptotics of the coefficients). As we shall see, the situation is not so simple, notably because the expression of the function \eqref{eq:def_Q} is fairly complicated to deal with, in particular when it is given under an integral form, which is a frequent situation in concrete case studies.

If we restrict question \ref{question-asymptotic} to the computation of asymptotic estimates for  the number of walks ending at the origin $(0,0)$ (these numbers are the coefficients of $Q(0,0,z)$, see \eqref{eq:def_Q}), a very recent paper \cite{DW} solved the problem for much more general random walks in $\mathbb{R}^d$, for any $d\ge1$: indeed, the authors obtain the leading term of the estimate, up to a coefficient involving two unknown harmonic functions. On the other hand, it is impossible to deduce from \cite{DW} the asymptotic behavior of the number of walks ending, for instance, at one axis (coefficients of $Q(1,0,z)$ and $Q(0,1,z)$), or anywhere (coefficients of $Q(1,1,z)$). 

Still with respect to question \ref{question-asymptotic}, it is worth mentioning  several interesting conjectures by Bostan and Kauers \cite{BK1,BK2}, which are proposed independently of the finiteness of the group.

The aim of this paper is to propose a uniform approach to answer question \ref{question-asymptotic}, which, as we shall see, works  for \emph{both} finite or infinite groups, and for walks not necessarily  restricted to excursions as in \cite{DW}.

First, Section \ref{sec:functional-equation} presents the basic functional equation,  which is the keystone of the analysis. Then, Section \ref{sec:simple-walk} illustrates our approach by considering the simple walk (the jumps of which are represented in Figure \ref{ExExEx}), which has a group of order $4$. In this case, we obtain an explicit expression of $Q(x,y,z)$,  allowing  to get exact asymptotics of the coefficients of $Q(0,0,z)$, $Q(1,0,z)$ (and by an obvious symmetry, also of $Q(0,1,z)$) and of $Q(1,1,z)$, by making a direct \emph{singularity analysis}. 
Section \ref{sec:generalization} sketches the main differences (in the analytic treatment) occurring between the simple walk and any of the $78$ other models. In particular, we explain the key phenomena leading to singularities. Indeed, a singularity is due to the fact that a certain Riemann surface sees its \emph{genus} passing from $1$ to $0$, while another singularity takes place at $z=1/\vert\Sc\vert$, where $|\Sc|$ denotes the number of possible steps of the model.  Also, as in the probabilistic context (although this will not be shown in this paper), the orientation of the \emph{drift vector} has a clear  impact on the nature of the singularities which lead to the asymptotics of interest.


To conclude (!) this introduction, let us mention that our work has clearly direct connections with probability theory. Indeed, with regard to random walks evolving in the quarter plane, and more generally in cones, the question of the asymptotic tail distribution of the first hitting time of the boundary has been for many decades of great interest. The approach proposed in this paper should hopefully lead to a complete solution of this problem for random walks with jumps to the eight nearest neighbors.

\section{The basic functional equation} \label{sec:functional-equation}
A  common starting point to study the $79$ walks presented in the introduction is to establish the following functional equation (proved in \cite{BMM}),  satisfied by the generating function defined in \eqref{eq:def_Q}:
     \begin{equation}
     \label{counting_func_eq}
          K(x,y,z)Q(x,y,z) = c(x)Q(x,0,z) + \widetilde{c}(y)Q(0,y,z)-\delta_{-1,-1} Q(0,0,z)-{xy}/{z},
     \end{equation}
where $c$ and $\widetilde c$ are defined in \eqref{def_abc} and 
\begin{equation} \label{def_kernel}
     K(x,y,z) = xy\textstyle[\sum_{-1\leq i,j \leq 1} \delta_{i,j}x^{i}y^{j}-{1}/{z}].
\end{equation}
Equation \eqref{counting_func_eq} holds at least in the region $\{|x|\leq 1,|y|\leq 1, |z|<1/|\Sc|\}$, since obviously $q(i,j,n)\leq {|\Sc|}^{n}$. In \eqref{counting_func_eq}, we note that the $\delta_{i,j}$'s defined in \eqref{def_deltaij} play the role (up to a normalizing condition)  of the usual transition probabilities $p_{i,j}$'s in a probabilistic context. 

Our main goal is to analyze the dependency of $Q(x,y,z)$ with respect to the time variable $z$, which in fact plays merely the role of a parameter, as far as the functional equation \eqref{counting_func_eq} is concerned. Remark that $Q(x,y,z)$ has real positive coefficients in its power series expansion. Thus, letting $R$ denote the radius of convergence, say of $Q(1,0,z)$, we know from Pringsheim's theorem (see \cite{TIT}) that $z=R$ is a singular point of this function. 

The quantity \eqref{def_kernel} that appears in \eqref{counting_func_eq} is called the {\em kernel} of the walk. It can be rewritten as 
     \begin{equation*}
          \widetilde{a}(y)x^{2}+[\widetilde{b}(y)-y/z]x+\widetilde{c}(y)=a(x)y^{2}+[b(x)-x/z]y+c(x),
     \end{equation*}
where 
     \begin{equation}
     \label{def_abc}
     \left.\begin{array}{lllllllll}
          \widetilde{a}(y)=\sum_{-1\leq j\leq 1}\delta_{1,j}y^{j+1},
&\widetilde{b}(y)=\sum_{-1\leq j\leq 1}\delta_{0,j}y^{j+1},
&\widetilde{c}(y)=\sum_{-1\leq j\leq 1}\delta_{-1,j}y^{j+1},\\
a(x)=\sum_{-1\leq i\leq 1}\delta_{i,1}x^{i+1},
&  b(x)=\sum_{-1\leq i\leq 1}\delta_{i,0}x^{i+1},
&c(x)=\sum_{-1\leq i\leq 1}\delta_{i,-1}x^{i+1}.
\end{array}\right.
\end{equation}
Let us also introduce the two discriminants of the kernel, in the respective $\C_x$ and $\C_y$ complex planes,
     \begin{equation}
     \label{def_d}
          \widetilde{d}(y,z)=[\widetilde{b}(y)-y/z]^{2}-4\widetilde{a}(y)\widetilde{c}(y),
          \qquad d(x,z)=[b(x)-x/z]^{2}-4a(x)c(x). 
     \end{equation}

From now on, we shall take  $z$ to be a \emph{real} variable. Indeed this is in no way a restriction, as it will emerge from analytic continuation arguments.

For $z\in(0,1/\vert \mathcal S\vert)$, the polynomial $d$ (resp.\ $\widetilde{d}$) has four roots (one at most being possibly infinite) satisfying in the $x$-plane (resp.\ $y$-plane)
\begin{equation}
\label{eq:bp1}
     |x_{1}(z)|<x_{2}(z)<1<x_{3}(z)<|x_{4}(z)|\leq \infty,
     \quad
     |y_{1}(z)|<y_{2}(z)<1<y_{3}(z)<|y_{4}(z)|\leq \infty,
\end{equation} 
as shown in \cite{Ra}. Then consider the algebraic functions $X(y,z)$ and $Y(x,z)$ defined by 
\begin{equation*}
     K(X(y,z),y,z) = K(x, Y(x,z),z) =0.
\end{equation*}
With the notations \eqref{def_abc} and \eqref{def_d}, we have
\begin{equation}
\label{expression_X_Y}
     X(y,z) = \frac{-\widetilde{b}(y)+\frac{y}{z}\pm\sqrt{\widetilde{d}(y,z)}}{2\widetilde a(y)},\qquad Y(x,z) = \frac{-b(x)+\frac{x}{z}\pm \sqrt{d(x,z)}}{2a(x)}.
\end{equation}
From now on, we shall denote by $X_0,X_1$ (resp.\ $Y_0,Y_1$) the two branches of these algebraic functions defined in the $\C_y$ (resp. $\C_x$)  plane. They can be separated (see \cite{FIM}),  to ensure $\vert X_0\vert\leq \vert X_1\vert$ in $\C_y$ (resp.\ $\vert Y_0\vert \leq \vert Y_1\vert$ in $\C_x$), keeping in mind that  the variable $z$ merely plays the role of a parameter.

\section{The example of the simple walk} 
\label{sec:simple-walk}
To illustrate the method announced in the introduction, we consider in this section the \emph{simple} walk (i.e.,  $\delta_{i,j} = 1$ only for couples $(i,j)$ such that $ij=0$, see Figure \ref{ExExEx}). Then the  following main results hold.
\begin{prop}
\label{prop:exact_simple_walk}
For the simple walk,
\begin{equation}\label{eq:Q001}
     Q(0,0,z) =\frac{1}{\pi} \int_{-1}^1 \frac{  1-2uz-\sqrt{(1-2uz)^2-4z^2}}{z^2}\sqrt{1- u^2}\,\textnormal{d}u.
\end{equation}
\end{prop}
Note that $Q(0,0,z)$ counts the number of excursions starting from $(0,0)$ and returning to $(0,0)$.
\begin{prop}
\label{prop:asymp-1_simple_walk}
For the simple walk, as $n\to\infty$, 
\begin{equation*}
     q(0,0,2n)\sim \frac{4}{\pi}\frac{16^{n}}{n^3}.
\end{equation*}
\end{prop}
\begin{prop}
\label{prop:Q(1,0,z)_simple_walk}
For the simple walk,
\begin{equation*}
     Q(1,0,z)=\frac{1}{2\pi}\int_{-1}^{1}\frac{1 - 2uz -\sqrt {(1-2uz)^2 - 4z^2}}{z^2}\sqrt{\frac{1+u}{1-u}}\, \textnormal{d}u.
\end{equation*}
\end{prop}
The generating function $Q(1,0,z)$ counts the number of walks starting from $(0,0)$ and ending at the horizontal axis. In addition, by an evident symmetry, $Q(0,1,z) = Q(1,0,z)$.
\begin{prop}
\label{prop:asymp-2_simple_walk}
For the simple walk, as $n\to\infty$, 
\begin{equation*}
     \sum_{i\geq 0}q(i,0,n)\sim \frac{8}{\pi}\frac{4^{n}}{n^2}.
\end{equation*}
\end{prop}
Having the expressions of $Q(1,0,z)$ and $Q(0,1,z)$, the series $Q(1,1,z)$ is directly obtained from \eqref{counting_func_eq}. The asymptotics of its coefficients, which represent the total number of walks of a given length, is the subject of the next result.
\begin{prop}
\label{prop:asymp-3_simple_walk}
For the simple walk, as $n\to\infty$, 
\begin{equation*}
     \sum_{i,j\geq 0}q(i,j,n)\sim \frac{4}{\pi}\frac{4^{n}}{n}.
\end{equation*}
\end{prop}

\begin{proof}[of Proposition \ref{prop:exact_simple_walk}]
The key point is to reduce  the computation of $Q(x,0,z)$  to a BVP set on $\Gamma =\{t\in\mathbb C : \vert t\vert =1\}$, according to the procedure proposed in \cite{FI,FIM}, and briefly described now.

 Letting $\D$ be a simply connected domain bounded by a smooth curve $\mathcal{L}$, the BVP we shall encounter in this study belongs to a wide class (studied among others by Riemann, Dirichlet, Hilbert, Carleman) and can be stated as follows. 
 
 \emph{Find a function $F$ of a single complex variable $x$, holomorphic in $\D$ and satisfying a boundary condition of the form}
\[
F(t) = F(\bar{t}) + g(t), \quad \forall t \in \mathcal{L},
\]
\emph{where $g(t)$ is a known function satisfying at least a  H\"older condition on $\mathcal{L}$}.

Indeed, for  any $t \in \Gamma$, we have
\begin{equation} \label{eq:BVP-circle}
     c(t)Q(t,0,z) - c(\bar{t})Q(\bar{t},0,z) = \frac{tY_0(t,z) - \bar{t}Y_0(\bar{t},z)}{z},
\end{equation}
where $\bar{t}$ stands for the complex conjugate of $t$. 
The proof of \eqref{eq:BVP-circle} relies on simple manipulations on the functional equation 
\eqref{counting_func_eq}. The main argument is the following. On $K(x,y,z)=0$, letting $y$ tend successively from above ($y^+$) and from below ($y^-$) to an arbitrary  point $y$ on  the cut 
$[y_1(z),y_2(z)]$ (which is merely a segment) defined in \eqref{eq:bp1}, $Q(0,y,z)$ remains continuous, and thus can be eliminated by computing the difference $Q(0,y^+,z)- Q(0,y^-,z) =0$ in the functional equation \eqref{counting_func_eq}, see \cite{FI,FIM} for full details.

In the present case, a pleasant point (from a computational point of view) is that, on the unit circle, we have $\bar{t}=1/t$. Hence, after mutiplying both sides of \eqref{eq:BVP-circle} by $1/(t-x)$ with $|x|<1$, integrating over $\Gamma$, and making use of Cauchy's formula, we get the  following integral form (noting that here $c(x)=x$)
\begin{equation} 
\label{eq:int-circle}
     c(x)Q(x,0,z) = \frac{1}{2i\pi z}\int_\Gamma \frac{tY_0(t,z) - \bar{t}Y_0(\bar{t},z)}{t-x}\,\text{d}t , \quad \forall |x|<1.
\end{equation}
For $t\in\Gamma$, we know from  general results \cite{FIM, Ra} that for $z \in [0,1/4]$, $Y_0(t,z)$ is real and belongs to the segment $[y_1(z),y_2(z)]$, the extremities of which are the two branch points located inside the unit circle in the $y$-plane, see Section \ref{sec:functional-equation}. Then it is not difficult to check that the integral on the right-hand side of \eqref{eq:int-circle} does vanish at $x=0$. Hence we are entitled to write
\begin{equation}
\label{eq:Q00}
     Q(0,0,z) = \lim_{x\to 0} \frac{1}{2i\pi x z}\int_\Gamma \frac{tY_0(t,z) - \bar{t}Y_0(\bar{t},z)}{t-x}\,\text{d}t = \frac{1}{2i\pi z}\int_\Gamma \frac{tY_0(t,z) - \bar{t}Y_0(\bar{t},z)}{t^2}\,\text{d}t,
\end{equation}
where the last equality follows from l'Hospital's rule.

To exploit formula \eqref{eq:Q00}, it is convenient  to make the straightforward change of variable $t=e^{i\theta}$, which yields 
\begin{equation} 
\label{eq:Q100}
     Q(0,0,z) = \frac{i}{2\pi z}\int_0^{2\pi} \sin\theta\, e^{-i\theta}Y_0(e^{i\theta},z) \,\text{d}\theta =  \frac{1}{2\pi z}\int_0^{2\pi} \sin^2\theta\, Y_0(e^{i\theta},z) \,\text{d}\theta,
\end{equation}
where we have used the fact that $Y_0(e^{i\theta},z)$ is real and $Y_0(e^{i\theta},z)=Y_0(e^{-i\theta},z)$. Furthermore, we have by \eqref{expression_X_Y} 
\begin{equation} \label{eq:Y0}
     Y_0(e^{i\theta},z)= \frac{1 -2z\cos \theta-\sqrt {(1-2z\cos\theta)^2 - 4z^2}}{2z}.
\end{equation}
Then, instantiating \eqref{eq:Y0} in \eqref{eq:Q100}, we get after some light algebra
 \begin{equation}\label{eq:Q200}
     Q(0,0,z) = \frac{1}{2z^2} - \frac{1}{\pi z^2}\int_0^{\pi} \sin^2\theta\sqrt {(1-2z\cos\theta)^2 - 4z^2} \,\text{d}\theta.
\end{equation}
Putting now $u=\cos\theta$ in the integrand, equation \eqref{eq:Q200} becomes exactly  
\eqref{eq:Q001}.
\end{proof}

\begin{proof}[of Proposition \ref{prop:asymp-1_simple_walk}] The function $Q(0,0,z)$ is holomorphic in $\mathbb C\setminus ((-\infty,-1/4]\cup[1/4,\infty))$, as it  easily emerges from Proposition \ref{prop:exact_simple_walk}. Accordingly, the asymptotics of its coefficients will be derived from the behavior of the function in the neighborhood of $\pm 1/4$ 
(see, e.g., \cite{FLAJ}). On the other hand $Q(0,0,z)$ is even, as can be seen in \eqref{eq:Q001}, or directly since $q(0,0,2n+1)=0$, for any $n\geq 0$,  in the case of the simple walk (see Figure \ref{ExExEx}).  Consequently, it suffices to focus on the point $1/4$.

\smallskip
First, we rewrite $\sqrt{(1-u^2)[(1-2uz)^2-4z^2]}$ as the product
\[
\sqrt{(1-u)(1-2(u+1)z)}\,\sqrt{(1+u)(1-2(u-1)z)},
\]
where the second radical admits the expansion 
$\sum_{i,j\geq 0}\mu_{i,j} (u-1)^i(z-1/4)^j$. Therefore, for $z$ in a neighborhood of $1/4$, we have
\begin{equation}
\label{Q(0,0,z)_after_expansion_1}
     Q(0,0,z) = -\frac{1}{\pi z^2}\sum_{i,j\geq 0}\mu_{i,j}(z-1/4)^j\int_{-1}^1(u-1)^i\sqrt{(1-u)(1-2(u+1)z)}\,\text du.
\end{equation}

We shall show below that, for any $i\geq 0$, there exist two functions, say $f_i$ and $g_i$, which are analytic at $z=1/4$ and such that, for $z$ in a neighborhood of $1/4$,
\begin{equation}
\label{Q(0,0,z)_after_expansion_2}
     \int_{-1}^1(u-1)^i\sqrt{(1-u)(1-2(u+1)z)}\,\text du = f_i(z)(z-1/4)^{i+2}\ln(1-4z)+g_i(z).
\end{equation}
Now, by inserting identity \eqref{Q(0,0,z)_after_expansion_2} into \eqref{Q(0,0,z)_after_expansion_1}, we obtain the existence of a function, say $g$, analytic at $z=1/4$ and such that
\begin{equation}
\label{Q(0,0,z)_after_expansion_3}
     Q(0,0,z) = -\frac{16}{\pi}\mu_{0,0}(z-1/4)^2\ln(1-4z)[f_0(1/4)+O(z-1/4)]+g(z).
\end{equation}
Furthermore, it is easy to prove $f_0(1/4)=4\sqrt{2}$ and $\mu_{0,0}=\sqrt{2}$. With these values, a classical singularity analysis (see, e.g.,\ \cite{FLAJ,TIT}) shows that, as $n\to\infty$, the $n{\text{th}}$ coefficient of the function in the right-hand side of \eqref{Q(0,0,z)_after_expansion_3} is equivalent to 
$(16/\pi) 4^n/n^3$. It is then immediate to infer that  
\begin{equation*}
     q(0,0,2n)\sim \frac{32}{\pi} \frac{4^{2n}}{(2n)^3}, 
\end{equation*}
remarking in the latter quantity the factor $32$ (and not $16\,!$), due to the parity of $Q(0,0,z)$ mentioned earlier.
\end{proof}

\begin{proof}[  of \eqref{Q(0,0,z)_after_expansion_2}]
Let $\Lambda_i(z)$ denote the integral in the left-hand side of \eqref{Q(0,0,z)_after_expansion_2}. The change of variable $u=1/(4z)-[(1-4z)/(4z)] v$ gives
\begin{equation*}
     \Lambda_i(z) = \left(\frac{1-4z}{4z}\right)^i\sqrt{2z}\int_{1}^{1+8z/(1-4z)}(1-v)^i\sqrt{v^2-1}\,
     \text{d}v.
\end{equation*}
Then, by letting $v = \cosh t$, we conclude that
\begin{equation}
\label{eq:after_cosh}
     \Lambda_i(z) = \left(\frac{1-4z}{4z}\right)^i\sqrt{2z}\int_{0}^{\cosh^{-1}(1+8z/(1-4z))}(1-\cosh t)^i(-1+\cosh^2 t)\,\text{d}t.
\end{equation}
Besides, by a classical linearization argument, there exist real coefficients $\alpha_0,\ldots ,\alpha_{i+2}$ such that 
\begin{equation}\label{eq:linearization}
     (1-\cosh t)^i(-1+\cosh^2 t) = \sum_{k=0}^{i+2}\alpha_k \cosh( kt).
\end{equation}
So, by means of  \eqref{eq:linearization}, it is easy to integrate $\Lambda_i(z)$. Then a delinearization argument shows the existence of two functions, $g_i$ and $h_i$, which are analytic at $z=1/4$ and satisfy
\begin{equation*}
     \Lambda_i(z) = g_i(z) +\left(\frac{1-4z}{4z}\right)^i \sqrt{2z}\,h_i(z) \cosh^{-1}(1+8z/(1-4z)).
\end{equation*}
Finally, remembering that, for all $u\ge1$,
$\cosh^{-1}u = \log (u + \sqrt{u^2-1})$, equation \eqref{Q(0,0,z)_after_expansion_2} follows, 
since $ \cosh^{-1}(1+8z/(1-4z))+\ln(1/4-z)$ is analytic at $z=1/4$.
\end{proof}

\begin{proof}[of Proposition \ref{prop:Q(1,0,z)_simple_walk}]
The argument mimics the one used in Proposition \ref{prop:asymp-1_simple_walk}. Instantiating $x=1$ in equation \eqref{eq:int-circle} and using \eqref{eq:Y0} yields directly
\begin{align*}
     Q(1,0,z)  &= \frac{1}{2i\pi z}\int_\Gamma \frac{tY_0(t,z) - \bar{t}Y_0(\bar{t},z)}{t-1}\,\text{d}t = 
     \frac{1}{2\pi z}\int_0^{2\pi} (e^{i\theta}+1)Y_0(e^{i\theta},z)\,\text{d}\theta \\[0.2cm]
   &= \frac{1}{2\pi z^2}\int_0^{\pi} (1+\cos \theta)(1 -2z\cos\theta-\sqrt {(1-2z\cos\theta)^2 - 4z^2}) \,\textnormal{d}\theta,
    \end{align*}
and the proof of Proposition \ref{prop:Q(1,0,z)_simple_walk} is complete.
\end{proof}

\begin{proof}[  of Proposition \ref{prop:asymp-2_simple_walk}]
It is quite similar to that of Proposition \ref{prop:asymp-1_simple_walk}, by starting from the integral formulation of $Q(1,0,z)$ written in Proposition \ref{prop:Q(1,0,z)_simple_walk}, and making an expansion of the integrand in the neighborhood of $u=1$ and $z=1/4$. So we omit it.
\end{proof}

\begin{proof}[ of Proposition \ref{prop:asymp-3_simple_walk}]
It is an immediate consequence of Proposition \ref{prop:asymp-2_simple_walk}, by using the equality 
\begin{equation*}
     (4-1/z)Q(1,1,z) = 2Q(1,0,z)-1/z,
\end{equation*}
which follows from \eqref{counting_func_eq}.
\end{proof}

\begin{rem} In fact $q(0,0,2n)$ is the product of the two Calalan's numbers $C_n$ and $C_{n+1}$ and similar expressions exist for the coefficients of $Q(1,0,z)$ and $Q(1,1,z)$. They can be obtained either using a shuffle of one dimensional-random walks, or the reflection principle 
(see for instance \cite{GKS}).
\end{rem}

\section[A general approch]{Asymptotics of the number of walks with small steps confined to the quarter plane: a general approch}
\label{sec:generalization}

This part aims at extending the approach of Section \ref{sec:simple-walk} to all $79$ models. For this, we show in Section \ref{sec-reduction} that $Q(x,0,z)$ and $Q(0,y,z)$ satisfy some BVP, that we solve in Section \ref{sec:CGF-and-BVP}. Next, in Section \ref{sec:computation} we compute $Q(0,0,z)$, $Q(1,0,z)$, $Q(0,1,z)$ and $Q(1,1,z)$ (question \ref{question-howmany}). In Section \ref{sec:group-4}, we see that the analysis of Section \ref{sec:simple-walk} applies verbatim to the $19$ walks having a group of order $4$. Finally, in Section \ref{sec:singularities} we analyze the singularities of the generating functions (question \ref{question-asymptotic}). In this short paper, we cannot provide the full details for the analysis of these singularities: we just explain where the singularities come from, and we postpone to the ongoing work \cite{FR3} the computation of the exact behavior of the generating functions in the neighborhood of the (dominant) singularities.
\subsection{Reduction to a boundary value problem}
\label{sec-reduction}
Most of the results in this section are borrowed from \cite{FIM,Ra}. In the general framework (i.e., for all $79$ models), equation \eqref{eq:BVP-circle} holds on the curve, drawn in the $x$-complex plane $\mathbb{C}$,
\begin{equation*}
     \mathcal{M}_z=  X_0(\underleftarrow{\overrightarrow{[y_1(z),y_2(z)]}},z) =\overline{X}_1(\underrightarrow{\overleftarrow{[y_1(z),y_2(z)]}},z),
\end{equation*}
 depicted in the next lemma, when the \emph{genus} of the Riemann surface (corresponding to the manifold $\{(x,y:\in\mathbb C^2:K(x,y,z)=0\}$) is equal to $1$. Here, remembering that $z$ is real, we let
 \begin{equation*}
      \underleftarrow{\overrightarrow{[y_1(z),y_2(z)]}}
\end{equation*}
stand for a contour, which is the slit $[y_1(z), y_2(z)]$ traversed from $y_1(z)$ to $y_2(z)$ along the upper edge, and then back to $y_1(z)$ along the lower edge. 
    \begin{lem}
     \label{Properties_curves_0}
        The curve $\mathcal M_z$ is one of the two components of a plane quartic curve  with  the following properties.
     \begin{itemize}
          \item It is symmetrical with respect to the real axis.
          \item It is  connected and closed in $\mathbb{C}\cup \{\infty\}$.
          \item It splits the plane into two connected domains, and we shall denote by 
          $\mathscr{G}(\mathcal{M}_z)$ that containing the point $x_{1}(z)$. 
          In addition, $\mathscr{G}(\mathcal{M}_z)\subset\mathbb{C}\setminus [x_{3}(z),x_{4}(z)]$.
     \end{itemize}
     \end{lem}
Similarly, we can define, in the $y$-complex plane $\mathbb{C}$, the curve $\mathcal{L}_z$ and the domain $\mathscr{G}(\mathcal{L}_z)$.

However, in the case of a group \eqref{def_generators_group} of order $4$, the two components introduced above coincide to form a double circle, as in Section \ref{sec:simple-walk}. Also, in the case of \emph{genus $0$}, there is only one component, which for instance can be an ellipse, see \cite[Chapter 6]{FIM}. 

Then,  starting from the formal boundary condition \eqref{eq:BVP-circle} on the curve $\mathcal{M}_z$, the function $Q(x,0,z)$ can be analytically continued with respect to $x$ from the unit disc to the domain $\mathscr{G}(\mathcal{M}_z)$, see \cite{FIM,Ra}. The properties quoted in this section lead to the formulation of the fundamental BVP. 

\emph{Find a function $x\mapsto Q(x,0,z)$ analytic in the domain $\mathscr{G}(\mathcal{M}_z)$, and satisfying  condition \eqref{eq:BVP-circle} on the boundary $\mathcal{M}_z$}.

\subsection{Solution of the boundary value problem by means of conformal gluing}
\label{sec:CGF-and-BVP}

To solve the BVP stated in the previous section, we make use of \cite{FI,FIM,Ra}. In particular, we need the notion of conformal gluing function (CGF). Indeed, as it is, the BVP holds on a curve that splits the plane into two connected components (see Lemma \ref{Properties_curves_0}). However, the BPV on the curve $\mathcal{M}_z$ can be reduced to a BVP set on a segment, which is somehow computationally more convenient (see \cite{LIT} and references therein). 
\begin{defn} \label{def_CGF}
     Let $\mathscr{C}\subset\mathbb{C}\cup \{\infty\}$ be an open and simply connected domain, symmetrical with respect to the real axis, and not equal to $\emptyset$, $\mathbb{C}$  or $\mathbb{C}\cup \{\infty\}$. A function $w$ is said to be a conformal gluing  function (CGF) for the domain $\mathscr{C}$ if it satisfies the following conditions.
     \begin{itemize}
     \item $w$ is meromorphic in $\mathscr{C}$.
     \item $w$ establishes a conformal mapping of $\mathscr{C}$ onto the complex plane cut along a segment.
     \item For all $t$ in the boundary of $\mathscr{C}$, $w(t)=w(\overline{t})$.
     \end{itemize}
\end{defn}

For instance, the mapping $t\mapsto t+1/t$ is a CGF for the unit disc centered at $0$. It is worth noting that the existence (however without any explicit expression) of a CGF for a generic domain $\mathscr{C}$ is ensured by general results on conformal gluing \cite[Chapter 2]{LIT}. In the sequel, we shall assume  that the unique pole  of $w$ is at $t=0$. Then the main result of this section is the following. \begin{prop}
\label{prop:expression_Q(x,0,z)}
For $x\in\mathscr{G}(\mathcal{M}_z)$,
\begin{equation} \label{expression_Q(x,0,z)}
c(x)Q(x,0,z)-c(0)Q(0,0,z)=\frac{1}{2\pi i z}\int_{\mathcal{M}_z} tY_0(t,z)\frac{w'(t,z)}{w(t,z)-w(x,z)}\,\textnormal{d}t,
\end{equation}
where $w(x,z)$ is the gluing function for the domain $\mathscr{G}(\mathcal{M}_z)$ in the 
$\C_x$-plane.

\end{prop}
Of course, a similar expression could be written for $Q(0,y,z)$.

\subsection{\protect Computation of $Q(0,0,z)$, $Q(1,0,z)$, $Q(0,1,z)$, $Q(1,1,z)$}
\label{sec:computation}
The expression we shall obtain for $Q(0,0,z)$ depends on $c(x)$ or symmetrically on 
$\widetilde c(y)$, see equation  \eqref{def_abc}, in the following respect.

\emph{(a)} Suppose first $c(0)=0$ (this equality holds for the simple walk). Then, as in Section \ref{sec:simple-walk}, we can write 
\begin{equation}
\label{expression_Q(0,0,z)_c(x)=0}
Q(0,0,z)=\lim_{x\to 0}\frac{1}{2\pi i zc(x)}\int_{\mathcal{M}_z} tY_0(t,z)\frac{w'(t,z)}{w(t,z)-w(x,z)}\,\textnormal{d}t.
\end{equation}

\emph{(b)} If $c(0)\neq 0$ and  $c(x)$ is not constant, then it has one or two roots, which are located on the unit circle $\Gamma$, see \eqref{def_abc} and \eqref{def_deltaij}. More precisely,
$c(x)$ takes either of the forms 
\[
x+1, \quad x^2+1, \quad x^2+x+1,
\]
with the respective roots $-1, \{i,\bar{i}\},  \{j,\bar{j}\}$. Let $\widehat x$ denote one of these roots. Then, provided that $\widehat{x}\in\mathscr{G}(\mathcal{M}_z)$, we can write
\begin{equation}
\label{expression_Q(0,0,z)_c(x)neq 0_c-not-constant}
Q(0,0,z)=-\frac{1}{2\pi i z}\int_{\mathcal{M}_z} tY_0(t,z)\frac{w'(t,z)}{w(t,z)-w(\widehat x,z)}\,\textnormal{d}t.
\end{equation}

\emph{(c)} If $c(0)\neq 0$ and $c(x)$ is constant, then we can evaluate the functional equation \eqref{counting_func_eq} at any point $(x,y)$ such that $\vert x\vert\leq 1$, $\vert y\vert\leq 1$ and $K(x,y,z)=0$, to obtain an expression for $Q(0,0,z)$.

The computation of $Q(1,0,z)$ and $Q(0,1,z)$ depends on the position of the point $1$ with respect to  $\mathscr{G}(\mathcal{M}_z)$ and $\mathscr G(\mathcal L_z)$. Indeed, if $1$ belongs to these domains,  then $Q(1,0,z)$ and $Q(0,1,z)$ are simply obtained by evaluating the integral formulations \eqref{expression_Q(x,0,z)} at $x=1$ and $y=1$. 

We now assume that for a given $z$, the point $1$ does not belong to $\mathscr{G}(\mathcal{M}_z)$. Then, by evaluating the functional equation \eqref{counting_func_eq} at $(x,Y_0(x,z))$ and 
 $(X_0(Y_0(x,z),z),Y_0(x,z))$, and by making  the difference of the two resulting relations, we obtain
\begin{equation}
\label{eq:after-two-fe}
     c(x)Q(x,0,z)= c(X_0(Y_0(x,z),z))Q(X_0(Y_0(x,z),z),0,z)+\frac{Y_0(x,z)}{z}[x-X_0(Y_0(x,z),z)].
\end{equation}
The key point is that, for $x\in\mathscr{G}(\mathcal{M}_z)$, the range of $\mathbb C$ through the composite function $X_0(Y_0(x,z),z)$ is $\mathscr{G}(\mathcal{M}_z)$ itself. This fact was proved in \cite[Corollary 5.3.5]{FIM} for $z=1/\vert \mathcal S\vert$, but the line of argument  easily extends to other values of $z$. In particular, to compute the right-hand side of \eqref{eq:after-two-fe}, we can use the expression \eqref{expression_Q(x,0,z)} valid  for any $x$, in particular for $x=1$.

As for $Q(1,1,z)$, we simply use the functional equation \eqref{counting_func_eq}, so that
\begin{equation}
\label{expression_Q(1,1,z)}
     (\vert \mathcal S\vert -1/z)Q(1,1,z) = c(1)Q(1,0,z)+\widetilde c(1)Q(0,1,z)-\delta_{-1,-1}Q(0,0,z)-1/z.
\end{equation}

\subsection{Case of the group of order $4$}
\label{sec:group-4}

For the $19$ walks with a group \eqref{def_generators_group} of order $4$ (the step sets 
$\mathcal S$ for these models have been classified in \cite{BMM}), there are two possible ways to compute the exact expression and the asymptotic of the coefficients of $Q(0,0,z)$, $Q(1,0,z)$, $Q(0,1,z)$ and $Q(1,1,z)$. 

Firstly, we can use Section \ref{sec:CGF-and-BVP}, as for any of the $79$ models. Secondly, the reasoning presented in Section \ref{sec:simple-walk} extends immediately  to the $19$ walks having a group of order $4$. Indeed, since in this case the boundary condition \eqref{eq:BVP-circle} is set on a circle, we can replace $\overline t$ by a simple linear fractional transform of $t$---for example, for the unit circle centered at $0$, $\overline t=1/t$. In addition, it is proved in \cite{FIM} that having a BVP set on a circle is equivalent  for the group to be of order $4$. 

For the simple walk, one can check that these two ways of approach coincide: take $t+1/t$ for the CGF $w(t,z)$ in \eqref{expression_Q(x,0,z)}, then make a partial fraction expansion of $w'(t,z)/[w(t,z)-w(x,z)]$, and this will eventually lead to \eqref{eq:int-circle}. 

\subsection{Singularities of the generating functions} \label{sec:singularities}

As written in the beginning of Section \ref{sec:generalization}, we cannot in this short paper go into deeper detail about the analysis of singularities. Let us simply state that only the real singularities of $Q(0,0,z)$, $Q(1,0,z)$, $Q(0,1,z)$ and $Q(1,1,z)$ play a role. Below, from the expressions of these generating functions obtained in Section \ref{sec:computation}, we just explain the main origin of all possible singularities, and postpone the complete proofs to \cite{FR3}.

Consider first $Q(0,0,z)$. 
\begin{prop}
\label{prop:singularity-Q(0,0,z)}
The smallest positive singularity of $Q(0,0,z)$ is
\begin{equation}
\label{def_z_g}
     z_g = \inf \{ z>0 : y_2(z) = y_3(z)\}.
\end{equation}
\end{prop}

\begin{rem}
\label{rem:singularity-Q(0,0,z)}
We choose to denote the singularity above by $z_g$, as an alternative definition of $z_g$ could be the following: the smallest positive value of $z$ for which the {\em genus} of the algebraic curve $\{(x,y)\in\mathbb C^2: K(x,y,z)=0\}$ jumps from $1$ to $0$. See also Proposition \ref{prop:characterizations-z_g} for other characterizations of $z_g$.
\end{rem}

\begin{proof}[Sketch of the proof of Proposition \ref{prop:singularity-Q(0,0,z)}]
To find the singularities of $Q(0,0,z)$, we start from the expressions obtained in Section \ref{sec:computation}, especially  \eqref{expression_Q(0,0,z)_c(x)=0} and \eqref{expression_Q(0,0,z)_c(x)neq 0_c-not-constant}. 

Consider first the case of the group of order $4$. Then the quartic curve $\mathcal{M}_z$ is a circle for any $z$ (see Section \ref{sec-reduction}), and $w(t,z)$ is a rational function of valuation $1$ (see Section \ref{sec:CGF-and-BVP}). In fact, the singularities come from $Y_0(t,z)$, since the branch points $x_\ell(z)$ appear in the expression of this function, see \eqref{def_d} and \eqref{expression_X_Y}. Hence, the first singularity of $Q(0,0,z)$ is exactly the smallest singularity of the $x_\ell(z)$, and this corresponds to equation \eqref{def_z_g},  see  Proposition \ref{prop:characterizations-z_g} below.

Consider now all the remaining cases (i.e., an infinite group, or a finite group of order strictly larger than $4$). Then the singularity $z_g$ appears not only for the same reasons as above, but also on account of the CGF $w(t,z)$. Indeed, for $z\in (0,z_g)$, the curve $\mathcal{M}_z$ is smooth, while for $z=z_g$ one can show that  $\mathcal{M}_z$ has a non-smooth double point at $X(y_2(z),z)$, see \cite{FIM,FR2}. Accordingly, $w(t,z)$ has a singularity at $z_g$, and the behavior of $w(t,z)$ in the neighborhood of $z_g$ is strongly related to the angle between the two tangents at the double point of the quartic curve, see \cite{FR2}.
\end{proof}

In Proposition \ref{prop:singularity-Q(0,0,z)} and in Remark \ref{rem:singularity-Q(0,0,z)}, we gave two different characterizations of $z_g$. We present hereafter five other ones.
\begin{prop}
\label{prop:characterizations-z_g}
The value of the first singularity $z_g$, introduced in \eqref{def_z_g}, can also be characterized by the five following equivalent statements.
\begin{enumerate}
     \item \label{tcqfp} $z_g = \inf \{ z>0 : x_2(z) = x_3(z)\} = \inf \{ z>0 : y_2(z) = y_3(z)\}$.
     \item $z_g$ is the smallest positive singularity of the branch points $x_\ell(z)$.
     \item $z_g$ is the smallest positive singularity of the branch points $y_\ell(z)$.
     \item \label{fps} $z_g$ is the smallest positive double root of the discriminant $d(x,z)$ considered as a polynomial in $x$.    
 \item \label{dw} $z_g$ is related to the minimizer of the Laplace transform of the $\delta_{i,j}$'s as follows. Define $(\alpha,\beta)$ as the unique solution in $(0,\infty)^2$ of
     \begin{equation}
     \label{eq:def-critical-point}
          \textstyle
          \sum_{-1\leq i,j\leq 1} i \delta	_{i,j} \alpha^i \beta^j = 0,
          \qquad
          \sum_{-1\leq i,j\leq 1} j \delta	_{i,j} \alpha^i \beta^j = 0.
     \end{equation}    
     Then 
     \begin{equation}
     \label{eq:z_g-DW}
          z_g = \frac{1}{\sum_{-1\leq i,j\leq 1} \delta_{i,j} \alpha^i \beta^j}.
     \end{equation}
\end{enumerate}
\end{prop}

Before proving Proposition \ref{prop:characterizations-z_g}, we notice at once that  $z_g$ (fortunately!) coincides with the smallest positive singularity of $Q(0,0,z)$ found in \cite[Section 1.5]{DW}, which does match \eqref{eq:def-critical-point} and \eqref{eq:z_g-DW}. 

Secondly, an easy consequence of any of the five points of Proposition \ref{prop:characterizations-z_g} is that $z_g$ is algebraic. Indeed, point \ref{fps} implies that this degree of algebraicity is at most $7$, and can in fact be strictly smaller than $7$, as implied by the following corollary of Proposition \ref{prop:characterizations-z_g} \ref{dw}.

\begin{coro}
\label{coro:drift-0}
We have $z_g = 1/\vert \mathcal S\vert$ if and only if $\sum_{-1\leq i\leq 1}i\delta_{i,j}=\sum_{-1\leq j\leq 1}j\delta_{i,j}=0$ (this happens for $14$ models, according to the classification proposed in \cite{BMM}). Otherwise, $z_g>1/\vert \mathcal S\vert$.
\end{coro}

\begin{proof}[of Proposition \ref{prop:characterizations-z_g}]
Only point \ref{dw} has to be shown, since the others are easy by-products of the definition 
\eqref{def_z_g} of $z_g$. For this purpose, we use (see \eqref{def_z_g} and Proposition 
\ref{prop:characterizations-z_g} \ref{tcqfp}) the equalities 
\begin{equation*}
     z_g = \inf \{ z>0 : y_2(z) = y_3(z)\} = \inf \{ z>0 : x_2(z) = x_3(z)\}.
\end{equation*}
Then we have (see, e.g., \cite[Chapter 6]{FIM} or \cite[Section 5]{KuRa})
\begin{equation*}
     X(y_2(z_g),z_g) = x_2(z_g),
     \qquad
     Y(x_2(z_g),z_g) = y_2(z_g).
\end{equation*}
Hence, the pair $(\alpha,\beta) = (x_2(z_g),y_2(z_g))$ satisfies the system   
\begin{equation*}
\label{eq:}
   K(\alpha,\beta,z_g) =  \frac{\partial}{\partial x} K(\alpha,\beta,z_g) = \frac{\partial}{\partial y} K(\alpha,\beta,z_g) = 0,
\end{equation*}
which in turn yields  \eqref{eq:def-critical-point}, while \eqref{eq:z_g-DW} is a direct consequence of $K(\alpha,\beta,z_g)=0$. To see why \eqref{eq:def-critical-point} has a unique solution in $(0,\infty)^2$, we refer for instance to \cite[Section 5]{KuRa}. 

Since claims \ref{dw} and \ref{tcqfp} are equivalent to the fact that both discriminants $d$ and 
$\widetilde{d}$ have a double root, the proof of the proposition is concluded.
\end{proof}

In \cite{FR3}, while proving Proposition \ref{prop:singularity-Q(0,0,z)}, we shall besides be able to obtain the precise behavior of $Q(0,0,z)$ near $z_g$. To this end, we use and extend results of \cite{FR2,Ra} to derive some fine properties of the CGF, when the Riemann surface passes  from genus $1$ to genus $0$.  

As for the singularities of $Q(1,0,z)$ and $Q(0,1,z)$, a key point is to locate the point $1$ with respect to the domains $\mathscr{G}(\mathcal{M}_z)$ and $\mathscr{G}(\mathcal{L}_z)$. Indeed, the expressions of $Q(1,0,z)$ and $Q(0,1,z)$ written in Section \ref{sec:computation} depend on this location.

If $1$ belongs to the sets above for any $z\in(0,z_g)$, then $Q(1,0,z)$ and $Q(0,1,z)$ are simply obtained by evaluating the integrals in Proposition \ref{prop:expression_Q(x,0,z)} at $x=1$ and $y=1$. The first singularity of these functions is then $z_g$ again. 

Assume now that for some $z\in(0,z_g)$, $1$ does not belong to $\mathscr{G}(\mathcal{M}_z)$. Then we use \eqref{eq:after-two-fe} to define $Q(1,0,z)$. Then the singularities of $Q(1,0,z)$ have to be sought among $z_g$ and the singularities of $Y_0(1,z)$ and $X_0(Y_0(1,z),z)$. But $X_0(Y_0(1,z),z)$ is either equal to $1$, or to $\widetilde c(Y_0(1,z))/\widetilde a(Y_0(1,z))$, see \cite[Corollary 5.3.5]{FIM}, and accordingly is either regular or has the same singularities as $Y_0(1,z)$.

Equations \eqref{def_abc}, \eqref{def_d} and \eqref{expression_X_Y} imply that the singularities of $Y_0(1,z)$ necessarily satisfy $d(1,z)=0$. As a consequence, the smallest positive singularity of $Y_0(1,z)$ is given by
\begin{equation*}
     z_{Y}=\frac{1}{b(1)+2\sqrt{a(1)c(1)}}.
\end{equation*} 

\begin{lem}
$z_Y\in[1/\vert\mathcal S\vert,z_g]$. 
\end{lem}

\begin{proof}
The inequality $z_Y\geq 1/\vert\mathcal S\vert$ is a consequence of the relations 
\[
2\sqrt{a(1)c(1)}\leq a(1)+c(1), \quad a(1)+b(1)+c(1) = \vert\Sc\vert. 
\]
Also, $z_Y$ satisfies $d(1,z_Y)=0$. In other words, $z_Y$ is the smallest positive value of $z$ such that $1$ is a root of $d(x,z)$. But we have $x_2(0)=0$, $x_3(0)=\infty$, and $x_2(z_g) = x_3(z_g)$, see Proposition \ref{prop:characterizations-z_g}. Hence, by a continuity argument, one of the points $x_2(z)$ or $x_3(z)$ must become $1$ before reaching  the other branch point.
\end{proof}

Similarly, we obtain that $Q(0,1,z)$ has a singularity at $z_g$, and possibly at 
\begin{equation*}
     z_{X}=\frac{1}{\widetilde b(1)+2\sqrt{\widetilde a(1)\widetilde c(1)}}.
\end{equation*} 

Concerning the algebraicity of $z_Y$ and $z_X$, we simply note that these numbers are either rational or algebraic of degree two, if and only if $a(1)c(1)$ and $\widetilde a(1)\widetilde c(1)$ are not square numbers.

As for $Q(1,1,z)$, once we know the singularities of $Q(0,0,z)$, $Q(1,0,z)$ and 
$Q(0,1,z)$, it is immediate  to compute those of $Q(1,1,z)$ by means of  equation \eqref{expression_Q(1,1,z)}.

\begin{rem}[Conclusion]
As we have seen, a consequence of the various results of Section \ref{sec:generalization} is that the smallest positive singularities of $Q(0,0,z)$, $Q(1,0,z)$, $Q(0,1,z)$ and $Q(1,1,z)$ are algebraic, and sometimes even rational, for all models of walks. In addition, we have the following classification of the singularities. For each of the $74$ non-singular models, let us introduce the \emph{mean drift vector} 
\[
\overrightarrow{M} = (M_x, M_y),
\]
with 
\begin{equation*}
    M_x =  \sum_{-1\leq i,j\leq 1}i\delta_{i,j}, \quad M_y = \sum_{-1\leq i,j\leq 1}j\delta_{i,j} ,
\end{equation*}
and the \emph{covariance} 
\begin{equation*}
     C = \sum_{-1\leq i,j\leq 1}ij\delta_{i,j} - M_x M_y,
     \end{equation*}
 where the $\delta_{i,j}$ have been introduced in \eqref{def_deltaij}.
The table below gives precisely the smallest positive singularities of each generating function $Q(1,0,z)$, $Q(0,1,z)$ and $Q(1,1,z)$ in terms of the sign of the coordinates of the drift vector and the covariance. We write ``FS'' to mean  ``first positive singularity''.

\setlength{\doublerulesep}{\arrayrulewidth}
\begin{center}
\begin{tabular}{|||c|||p{8mm}|p{8mm}|p{8mm}|||p{8mm}|p{8mm}|p{8mm}|||p{8mm}|p{8mm}|p{8mm}|p{8mm}|||}
\hline\hline\hline
& \multicolumn{3}{|c|||}{\text{FS of } $Q(1,0,z)$:} & \multicolumn{3}{|c|||}{\text{FS of } $Q(0,1,z)$:}& \multicolumn{4}{c|||}{\text{FS of } $Q(1,1,z)$:} \\ 
  \hline
\text{Drift:}& \multicolumn{1}{|c|}{$z_g$} & \multicolumn{1}{|c|}{$z_Y$} & \multicolumn{1}{|c|||}{$1/\vert \mathcal S\vert $} & \multicolumn{1}{|c|}{$z_g$} & \multicolumn{1}{|c|}{$z_X$} & \multicolumn{1}{|c|||}{$1/\vert \mathcal S\vert$} & \multicolumn{1}{|c|}{$z_g$} & \multicolumn{1}{|c|}{$z_X$} & \multicolumn{1}{|c|}{$z_Y$} & \multicolumn{1}{|c|||}{$1/\vert\mathcal S\vert$} \\
  \hline\hline\hline
  $(+,+)$&  & \multicolumn{1}{|c|}{ {\large $\times$}}& & & \multicolumn{1}{|c|}{ {\large $\times$}} & & & &  & \multicolumn{1}{|c|||}{ {\large $\times$}} \\
  \hline
    $(+,0)$&  & \multicolumn{2}{|c|||}{ {\large $\times$}} &\multicolumn{1}{|c|}{{\footnotesize $C\geq0$}}  &\multicolumn{1}{|c|}{{\footnotesize $C\leq0$}}   & & & &  & \multicolumn{1}{|c|||}{ {\large $\times$}} \\
  \hline
    $(0,+)$&\multicolumn{1}{|c|}{{\footnotesize $C\geq0$}}  & \multicolumn{1}{|c|}{{\footnotesize $C\leq0$}} & & & \multicolumn{2}{|c|||}{ {\large $\times$}} & & &  &\multicolumn{1}{|c|||}{ {\large $\times$}} \\
  \hline
    $(0,0)$& \multicolumn{3}{|c|||}{ {\large $\times$}} & \multicolumn{3}{|c|||}{ {\large $\times$}} & \multicolumn{4}{|c|||}{ {\large $\times$}} \\
  \hline
    $(+,-)$&  & \multicolumn{1}{|c|}{ {\large $\times$}} & & \multicolumn{1}{|c|}{ {\large $\times$}}&  & & & & \multicolumn{1}{|c|}{ {\large $\times$}} &  \\
  \hline
    $(-,+)$& \multicolumn{1}{|c|}{ {\large $\times$}} &  & & & \multicolumn{1}{|c|}{ {\large $\times$}} & & &\multicolumn{1}{|c|}{ {\large $\times$}} &  & \\
  \hline
    $(0,-)$& \multicolumn{1}{|c|}{{\footnotesize $C\leq0$}} & \multicolumn{1}{|c|}{{\footnotesize $C\geq0$}} & & \multicolumn{1}{|c|}{ {\large $\times$}} &  &  &  \multicolumn{1}{|c|}{{\footnotesize $C\leq0$}} & \multicolumn{1}{|c|}{{\footnotesize $C\geq0$}} &  & \\
  \hline
    $(-,0)$& \multicolumn{1}{|c|}{ {\large $\times$}}   &  & &   \multicolumn{1}{|c|}{{\footnotesize $C\leq0$}}  &   \multicolumn{1}{|c|}{{\footnotesize $C\geq0$}}  & & \multicolumn{1}{|c|}{{\footnotesize $C\leq0$}} & \multicolumn{1}{|c|}{{\footnotesize $C\geq0$}} &  & \\
  \hline
    $(-,-)$& \multicolumn{1}{|c|}{ {\large $\times$}}  &  & & \multicolumn{1}{|c|}{ {\large $\times$}}  &  & & \multicolumn{1}{|c|}{ {\large $\times$}}  & &  & \\
  \hline\hline\hline
\end{tabular}
\end{center}
\end{rem}
\acknowledgments
This paper is dedicated to the memory of  Philippe Flajolet, who was deeply interested in these problems, about which we had together enjoyable preliminary discussions. This  brutal stroke of fate prevented our friend Philippe from becoming the third co-author\ldots 

\bibliographystyle{abbrvnat}

\end{document}